\documentclass[conference]{IEEEconf}
\IEEEoverridecommandlockouts
\usepackage[noadjust]{cite}
\usepackage{amsmath,amssymb,amsfonts}
\usepackage{algorithmic}
\usepackage{graphicx}
\usepackage{textcomp}
\usepackage{xcolor}
\usepackage{dsfont}
\usepackage{amsthm}

\let \labelindent \relax
\usepackage{enumitem}
\usepackage{subcaption}
\usepackage[font=small]{caption}

\usepackage{accents}

\newcommand\munderbar[1]{%
 \underaccent{\bar}{#1}}

\newtheorem{theorem}{Theorem}[section]
\newtheorem{corollary}{Corollary}[theorem]
\newtheorem{lemma}[theorem]{Lemma}

\newtheorem{proposition}[theorem]{Proposition}
\newtheorem*{assumption}{Assumption}

\newcommand{\Na}{{N_{\rm a}}}

\def\BibTeX{{\rm B\kern-.05em{\sc i\kern-.025em b}\kern-.08em
    T\kern-.1667em\lower.7ex\hbox{E}\kern-.125emX}}
\begin{document}

\title{Control of Agreement and Disagreement Cascades with Distributed Inputs
\thanks{Supported by ONR grant N00014-19-1-2556, ARO grant W911NF-18-1-0325, DGAPA-UNAM PAPIIT grant IN102420, and Conacyt grant A1-S-10610, and by  
NSF Graduate Research Fellowship  
DGE-2039656.}
}

\author{Anastasia~Bizyaeva, Timothy~Sorochkin,
        Alessio~Franci,
        and~Naomi~Ehrich~Leonard
\thanks{A. Bizyaeva and N.E. Leonard are with the Department
of Mechanical and Aerospace Engineering, Princeton University, Princeton,
NJ, 08544 USA; e-mail: bizyaeva@princeton.edu, naomi@princeton.edu.}
\thanks{T. Sorochkin is with the Department of Physics \& Astronomy, University of Waterloo, Waterloo, ON, N2L 3G1 Canada; e-mail: tsorochkin@uwaterloo.ca }
\thanks{A. Franci is with the Department of Mathematics, National Autonomous University of Mexico, 04510 Mexico City, Mexico. e-mail: afranci@ciencias.unam.mx}}

\maketitle

\begin{abstract}
For a group of autonomous communicating agents, the ability to distinguish a meaningful input from disturbance, and come to collective agreement or disagreement in response to that input, is paramount for carrying out coordinated objectives. In this work we study how a cascade of opinion formation spreads through a group of networked decision-makers in response to a distributed input signal. Using a nonlinear opinion dynamics model with dynamic feedback modulation of an attention parameter, we show how the triggering of an opinion cascade and the collective decision itself depend on both the distributed input and the node agreement and disagreement centrality, determined by the spectral properties of the network graph. We further show how the attention dynamics introduce an implicit threshold that distinguishes between distributed inputs that trigger cascades and ones that are rejected as disturbance.

\end{abstract}


\section{Introduction}

Emerging technologies rely on network communications and sensor input in order to make coherent collective decisions. For example, autonomous multi-robot teams must cooperate to move as a group, avoid obstacles and one another, and perform collective tasks in potentially dynamic and uncertain environments. These objectives necessarily involve on-the-fly collective decision-making about context-dependent options, such as which of multiple available paths to take, which direction to turn, or how to distribute available tasks among the team members. There is urgent need for a unified design framework that enables autonomous teams to coordinate  decisions across different contexts in a distributed manner. Mathematical models of networked opinion dynamics, e.g. \cite{Peng2015,Parsegov2017,Fontan2018,Basar2018,Valcher2020}, are particularly useful for this purpose, in part due to their relative simplicity and analytic tractability. 

With this as motivation, we focus in this paper on a general model of distributed opinion formation on a network recently introduced in \cite{Bizyaeva2020, Franci2020}. In this multi-agent, multi-option framework, opinion formation is treated as a nonlinear process in which agents update their real-valued opinions in continuous time in response to saturated information exchanges from their communication network. A prominent feature of this model is the emergence of consensus and dissensus solutions in easily identifiable parameter regimes, even when all of the agents are homogeneous and the communication network is all-to-all. These analytic results in \cite{Bizyaeva2020} are independent of the number of agents or options. Analysis in \cite{BizyaevaACC2020}  suggests that emergence of agreement and disagreement, through bifurcations, persists across different network topologies. This makes our model flexible and  adaptable to different contexts and applications. The spectral properties of the network adjacency matrix are key to the agreement and disagreement bifurcation structure. The bifurcation kernel is generated by the dominant eigenvector in the agreement regime and by the eigenvector associated to the smallest eigenvalue in the disagreement regime. The entries of these eigenvectors define agreement and disagreement centrality indices for the nodes of the network in symmetric graphs~\cite{Franci2021}. For digraphs, node centrality is defined by the left adjacency matrix eigenvectors associated to the largest (agreement case) or smallest (disagreement case) eigenvalues.

We exploit the connection between the spectral properties of the network adjacency matrix and the agreement and disagreement bifurcation behavior to derive 
criteria 
for designing distributed inputs to control opinion formation and opinion cascades. Using Lyapunov-Schmidt reduction methods~\cite[Chapter VII]{Golubitsky1985}, we show that the agreement (disagreement) centrality of a node determines how 
an input to it affects the agreement (disagreement) bifurcation behavior. Our results rigorously justify 
the agreement and disagreement centralities introduced in~\cite{Franci2021}. When the opinion dynamics are coupled with the feedback attention dynamics introduced in~\cite{Bizyaeva2020}, sufficiently large inputs can trigger an opinion cascade at which the network state transitions from a weakly opinionated state to a strongly opinionated state. Agreement and disagreement centrality indices predict sensitivity of opinion cascades to distributed inputs.
The more aligned the input vector is with the centrality vector, the smaller the  inputs need to be to trigger a cascade.

We present the model Section~\ref{sec:model} and review Lyapunov-Schmidt reduction in Section~\ref{sec:LS}. We prove the role of distributed input on opinion formation behavior for constant attention in Section~\ref{sec:constantu} and for dynamic feedback controlled attention in Section~\ref{sec:dynamicu}. 



\section{Opinion Dynamics Model} 
\label{sec:model}
    
We study a model of $N_a$ agents communicating over a network and forming opinions on two options through a nonlinear process specialized from the multi-option general model in \cite{Bizyaeva2020},\cite{Franci2020}. As in \cite{BizyaevaACC2020}, we specialize to agents that are homogeneous with respect to three fixed parameters in the dynamics: the rate of forgetting (damping coefficient $d >0$), the edge weight in the communication network ($\gamma \in \mathds{R}$), and the strength of self-reinforcement of opinion ($\alpha \geq 0$). In \cite{BizyaevaACC2020}, we focused on the zero-input setting, i.e., the case in which there is no stimulus, evidence or bias that informs the agents about the relative merits of the options.  Instead, here, we consider an input $b_i \in \mathds{R}$, for each agent $i = 1, \ldots, N_a$, and allow the inputs to be heterogeneously distributed over the network of agents.  We further model heterogeneity over the agents in their attention to network exchange.  

The topology of agent interactions is encoded in a graph $G = (V,E)$ where $V = \{1, \dots, N_{a} \}$ is an index set of vertices. Each vertex $i \in V$ represents an agent, and $E \subseteq V \times V$ is the set of edges, which represent interactions between agents. We define the unweighted graph adjacency matrix $A = (a_{ik})$, $i,k \in V$, with elements satisfying $a_{ik} = 1$ if and only if $e_{ik} \in E$, and $a_{ik} = 0$ otherwise. We consider this matrix without self-loops, $a_{ii} = 0$ for all $i \in V$. $G$ is an \textit{undirected} graph 
if $a_{ik} = a_{ki}$ for all $i,k \in V$. Let $\lambda_i$, $i = 1, \ldots, N_a$, be the eigenvalues of $A$ and $W(\lambda_i)$  the generalized eigenspace associated to $\lambda_i$. We define $\lambda_{max}$ 
and $\lambda_{min}$ to be the $\lambda_i$ with largest and smallest real parts, respectively, and $\mathbf{v}_{max}$ and $\mathbf{v}_{min}$ to be the corresponding unit left eigenvectors.


In the two-option setting, we represent the opinion of each agent $i$ by a real-valued $x_i \in \mathds{R}$. The sign of $x_i$ corresponds to agent $i$ favoring option 1 ($x_{i} > 0$) or favoring option 2 ($x_i < 0$). The magnitude of the opinion variable $x_i$ describes the strength of agent $i$'s commitment to an option. We say an agent is \textit{opinionated} when $|x_{i}| > \vartheta > 0$ for some threshold $\vartheta$, and it is \textit{unopinionated} otherwise. Note that we introduce the threshold $\vartheta$ in order to distinguish a small opinion from a large opinion in our discussion of the results, and the threshold value itself will not play a role in the analysis (and can be made appropriately small in each case). The vector of agents' opinions $\mathbf{x} = (x_{1}, \dots, x_{N_{a}}) \in \mathds{R}^{N_{a}}$ is the {\em network opinion state}. When all agents are unopinionated the group is in an {\em unopinionated state}; a special case is the \textit{neutral state} $\mathbf{x} = \mathbf{0}$. 
We distinguish between key opinionated network states: 
\begin{itemize}
    \item Any pair of agents $i,k \in V$ \textit{agree} when they are opinionated and favor the same option, i.e. $\operatorname{sign}(x_{i}) = \operatorname{sign}(x_{k})$. The group is in an \textit{agreement state} when this property holds for all $i,k \in V$. 
    \item Any pair of agents $i,k \in V$ \textit{disagree} when they are opinionated and favor different options, i.e. $\operatorname{sign}(x_{i}) \neq \operatorname{sign}(x_{k})$.  The group is in a \textit{disagreement state} when this property holds for at least one pair of agents $i,k \in V$. 
\end{itemize}

Each agent updates its own opinion state in continuous time according to the nonlinear update rule: 
\begin{equation}
    \dot{x}_{i} = - d x_{i} +  u_{i} S\left(\alpha x_{i} + \gamma \sum_{\substack{k=1 \\ k \neq i}}^\Na a_{ik} x_{k}\right) + b_{i}  \label{eq:opinion_dynamics}.
\end{equation}
The rule has four parts: a damping term with coefficient $d>0$, a nonlinear interaction term that includes inter-agent exchanges with weight $\gamma \in \mathds{R}$, an opinion self-reinforcement term with weight $\alpha \geq 0$, and an additive input $b_i \in \mathds{R}$. 

The nonlinearity applied to the inter-agent exchanges and self-reinforcement is defined by an odd sigmoidal saturating function $S$ which satisfies $S(0)=0$, $S'(0)=1$, and $\operatorname{sign}(S''(z)) = -\operatorname{sign}(z)$. This is motivated from the literature 
and means that agent $i$ is more influenced by opinion fluctuations in its neighbors when their average opinion is close to neutral, and as neighbors' opinions grow large on average their influence levels off. In simulations and analysis throughout this paper we use $S = \tanh$. We purposely leave the sigmoid more general in the definition of the opinion dynamics \eqref{eq:opinion_dynamics} because the results in this paper generalize to arbitrary odd sigmoidal functions with minor modifications in the algebraic details of the proofs.






In the following proposition we specialize a result from \cite{BizyaevaACC2020} which serves as a starting point for our analysis.

\begin{proposition}[\cite{BizyaevaACC2020}, Theorem 1]
The following hold true for \eqref{eq:opinion_dynamics} with $u_{i} := u \geq 0$ and $b_{i} = 0$ for all $i = 1, \dots, N_{a}$: \\
\textbf{A. Cooperation leads to agreement:}  Let $G$ be a connected undirected graph. If $\gamma > 0 $, the neutral state $\mathbf{x} = 0$ is a locally exponentially stable equilibrium for $0 < u < u_{a}$ and unstable for $u >u_{a}$, with
    \begin{equation}
        u_{a} = \frac{d}{\alpha + \gamma \lambda_{max}}. \label{eq:u_agree}
    \end{equation}
    At $u = u_{a}$, branches of agreement equilibria, $x_{i} \neq 0$, $\operatorname{sign}(x_{i}) = \operatorname{sign}(x_{k})$ for all $i,k \in V$, emerge in a steady-state bifurcation off  of $\mathbf{x} = \mathbf{0}$
    along $W(\lambda_{max})$; \\
\textbf{B. Competition leads to disagreement:} Let $G$ be a connected undirected graph. If $ \gamma < 0 $ the neutral state $\mathbf{x} = \mathbf{0}$ is a locally exponentially stable equilibrium for $0 < u < u_{d}$ and unstable for $u >u_{d}$, with
    \begin{equation}
        u_{d} = \frac{d}{\alpha +\gamma \lambda_{min}} . \label{eq:u_disagree}
    \end{equation}
    At $u = u_{d}$, branches of disagreement equilibria, $\operatorname{sign}(x_{i}) = - \operatorname{sign}(x_{k})$ for at least one pair $i,k \in V $, $i \neq k$, emerge in a steady-state bifurcation off  of $\mathbf{x} = \mathbf{0}$ 
    along $W(\lambda_{min})$.
\label{propACC}
\end{proposition}

\section{Lyapunov-Schmidt Reduction}
\label{sec:LS}

To systematically characterize the equilibria of the opinion dynamics model as a function of parameters, we leverage the Lyapunov-Schmidt reduction and its use in computing bifurcation diagrams. Consider the $n-$dimensional dynamical system $\dot{\mathbf{y}} = \Phi(\mathbf{y},\mathbf{p})$, where  $\Phi: \mathds{R}^n \times \mathds{R}^m \rightarrow \mathds{R}^n$ is a smooth parameterized vector field, $\mathbf{y} \in \mathds{R}^{n}$ is a state vector, 
and $\mathbf{p} \in \mathds{R}^{m}$ is a vector of parameters. Let the $k^{th}$ order derivative of $\Phi$ at $(\mathbf{y},\mathbf{p})$ be
\begin{multline}
    (d^{k}\Phi)_{\mathbf{y},\mathbf{p}}(\mathbf{v}_{1}, \dots, \mathbf{v}_{k}) \\
    = \left.\frac{\partial}{\partial t_1} \dots \frac{\partial}{\partial t_k} \Phi\left(\mathbf{y} + \sum_{i = 1}^{k} t_{i} \mathbf{v}_{i}, \mathbf{p} \right)\right|_{t_{1} = \dots = t_{k} = 0}.
\end{multline}
The equilibria of  $\dot{\mathbf{y}} = \Phi(\mathbf{y},\mathbf{p})$  are the level sets $\Phi(\mathbf{y},\mathbf{p}) = 0$, which defines the \textit{bifurcation diagram} of the system. 

The Jacobian of the system is the matrix $J$ with elements $J_{ij} = \frac{\partial  \Phi(\mathbf{y},\mathbf{p})}{ \partial y_{ij}}$. When $J$ evaluated at an equilibrium point $(\mathbf{y}^*,\mathbf{p}^*)$ is degenerate (i.e. has rank $n - m$ where $0 < m < n$), the local bifurcation diagram can be described using $m$ variables and the point is a \textit{singular point}. The \textit{Lyapunov-Schmidt reduction} of $\Phi(\mathbf{y},\mathbf{p})$ is an $m$-dimensional system of equations that captures the structure of the local bifurcation diagram of the system near $(\mathbf{y}^*,\mathbf{p}^*)$. The procedure for deriving the Lyapunov-Schmidt reduction involves projecting the Taylor expansion of $\Phi(\mathbf{y},\mathbf{p})$ onto the kernel of its Jacobian at the singularity. The Implicit Function Theorem is used to solve for $n-m$  variables as function of the remaining $m$, thus approximating the local vector field 
in the directions orthogonal to the kernel. For details on  Lyapunov-Schmidt reduction see \cite[Chapter VII]{Golubitsky1985}.

The \textit{normal form} for a bifurcation is the simplest equation that captures all  qualitative features of the bifurcation diagram. Dynamical systems with an odd state symmetry $\Phi(-\mathbf{y},\mathbf{p}) = - \Phi(\mathbf{y},\mathbf{p})$ often exhibit a \textit{pitchfork bifurcation}. A normal form for a pitchfork bifurcation universal unfolding is 
\begin{equation}
    \dot{y} = p_1 y \pm y^{3} + p_2 + p_3 y^2 \label{eq:NF_unfolded_pitchfork}
\end{equation}
where $y\in \mathds{R}$ is the reduced state, $p_1$ is a \textit{bifurcation parameter} and $p_2$, $p_3$ are \textit{unfolding parameters}. When $p_2 = p_3 = 0$, the symmetric pitchfork normal form is recovered in \eqref{eq:NF_unfolded_pitchfork}. When one of the unfolding parameters is nonzero, it follows from unfolding theory \cite[Chapter III]{Golubitsky1985} that the bifurcation diagram changes locally to one of four possible topologically distinct configurations (see Fig. \ref{fig:bif}). 
We use the tools outlined in this section to rigorously study how solutions of the distributed dynamics \eqref{eq:opinion_dynamics} depend on each agent's attention and input.

\section{Constant Attention: Sensitivity to Input near Critical Point}
\label{sec:constantu}

In this section, we investigate how a vector of constant inputs $\mathbf{b}$ informs the outcome of the opinion formation process \eqref{eq:opinion_dynamics} 
when attention is constant and $u_i := u \in \mathds{R}$ for all $i = 1, \dots, N_a$. The Jacobian  of \eqref{eq:opinion_dynamics} evaluated at  $\mathbf{x} = 0$ is 
\begin{equation}
    J_x =  (u \alpha - d) \mathcal{I} + u \gamma A \label{eq:jac}
\end{equation}
with identity matrix $\mathcal{I}$. 
The dynamics \eqref{eq:opinion_dynamics} in vector form are 
\begin{equation}
    \dot{\mathbf{x}}= - d \mathbf{x} + u \mathbf{S}\left( (\alpha \mathcal{I} + \gamma A) \mathbf{x}\right) + \mathbf{b}:= F(\mathbf{x},u,\mathbf{b}) \label{eq:dynamics_vectors}
\end{equation}
where  $\mathbf{S}(\mathbf{y}) = (S(y_{1}), \dots, S(y_{n}))$, $\mathbf{y} \in \mathds{R}^{n}$, and $\mathbf{b} = (b_{1}, \dots, b_{N_{a}})$. The following theorem generalizes results in \cite[Theorem 1]{Gray2018} to describe 
bifurcations of the opinion dynamics of homogeneous agents. The theorem shows that any bifurcation of $\mathbf{x} = 0$ of \eqref{eq:opinion_dynamics} that is generated by a simple eigenvalue of the adjacency matrix $A$ must be a pitchfork bifurcation. 

\begin{theorem}[Pitchfork Bifurcation] \label{thm:pitchfork} Consider \eqref{eq:opinion_dynamics} and define $u^{*} = \frac{d}{\alpha + \lambda \gamma }$, where $\lambda$ is a simple real eigenvalue of the adjacency matrix $A$ for a strongly connected graph $G$ with corresponding right unit eigenvector $\mathbf{v} = (v_1, \dots, v_{N_a})$ and corresponding left unit eigenvector $\mathbf{w} = (w_1, \dots, w_{N_a})$. Assume that  (i) for all eigenvalues $\xi \neq \lambda$ of $A$, $\operatorname{Re}[\xi] \neq \lambda$; (ii) $\alpha + \lambda \gamma  \neq 0$. Let $f(z,u,\mathbf{b})$ be the Lyapunov-Schmidt reduction of $F(\mathbf{x},u,\mathbf{b})$ at $(0,u^*,0)$ and assume $\langle \mathbf{w}, \mathbf{v}^3 \rangle \neq 0$. \\
\textbf{A.} The bifurcation problem $f(z,u,0) = 0$ has a symmetric pitchfork singularity at $(z, u,\mathbf{b}) = (0,u^*,0)$. For values of $u > u^*$ and sufficiently small $|u - u^*|$, two branches of equilibria branch off from $\mathbf{x} = 0$ in a pitchfork bifurcation along a manifold which is tangent at $\mathbf{x} = 0$ to $\operatorname{span}\{ \mathbf{v}\}$. When $\operatorname{sign}\{ \langle \mathbf{w}, \mathbf{v}^3 \rangle/ \langle \mathbf{w}, \mathbf{v} \rangle \} (\alpha + \lambda \gamma)> 0$ ($< 0$) the bifurcation happens supercritically (subcritically) with respect to  $u$. \\
\textbf{B.} The bifurcation problem $f(z, u,\mathbf{b}) = 0$ is an $N_{a}$-parameter unfolding of the symmetric pitchfork, and $\frac{\partial f}{\partial b_{i}}(z,u,\mathbf{b}) = w_{i}$.  
\end{theorem}
\begin{proof}
The eigenvalues of  $J_{x}$ \eqref{eq:jac} are $ \mu = u \alpha - d + u \gamma \lambda $, and so at $u = u^{*}$  $J_{x}$ has a single zero eigenvalue. Observe that the left and right null eigenvectors of $J_x$ are precisely $\mathbf{w}$ and $\mathbf{v}$.  
Following the procedure outlined in \cite[ Chapter I, 3.(e)]{Golubitsky1985} we derive $f(z,u,\mathbf{b})$. In particular we derive the coefficients of the polynomial expansion of $f(z,u,\mathbf{b})$ \cite[Chapter I, Equations 3.23(a)-(e)]{Golubitsky1985} through third order in the state variable. Note that $(d^{2} F)_{0,u^*,0}(\mathbf{v}_{1},\mathbf{v}_{2},\mathbf{v}_{3}) = 0$ for any $\mathbf{v}_{i}$ because $S''(0) = 0$, which implies that $f_{zz} = 0$ by \cite[Chapter I, Equation 3.23(b)]{Golubitsky1985}. Additionally, $f_{z}(0,u^*,0) = 0$ by \cite[Chapter I, Equation 3.23(a)]{Golubitsky1985}. The nonzero coefficients in the expansion read
\begin{gather*}
    f_{xxx} = \langle \mathbf{w}, (d^3 F)_{0,u^*,0}(\mathbf{v},\mathbf{v},\mathbf{v}) \rangle  = -2 d (\alpha + \lambda \gamma)^{2} \langle \mathbf{w}, \mathbf{v}^3 \rangle \\
    f_{b_{i}} = \left\langle \mathbf{w}, \frac{\partial F}{\partial b_{i}}(0,u^*,0) \right\rangle = w_{i} \\
    f_{\hat{u} x} = \left\langle \mathbf{w}, \left( d \frac{\partial F}{\partial \hat{u}}\right)_{0,u^*,0}(\mathbf{v},\mathbf{v})\ \right \rangle = (\alpha + \lambda \gamma) \langle \mathbf{w}, \mathbf{v}\rangle 
\end{gather*}
where $\hat{u} = u - u^{*}$ and $\langle \cdot, \cdot \rangle$ denotes the standard vector inner product. Additionally, observe that we can align the left and right eigenvectors to satisfy $\langle \mathbf{w}, \mathbf{v} \rangle = k_1 > 0$ (the inner product is nonzero by duality). Then $\langle \mathbf{w}, \mathbf{v}^3 \rangle := k_2 = \sum_{i = 1}^{N_a} w_i v_i^3 $. 
The Lyapunov-Scmidt reduction of \eqref{eq:opinion_dynamics} about $(0,u^*,0)$ is thus
\begin{equation}
    \dot{z} = k_1 (\alpha +\lambda  \gamma) \hat{u}z - 2 k_2 d (\alpha + \lambda \gamma)^{2} z^{3} + \langle \mathbf{w}, \mathbf{b}  \rangle + h.o.t. \label{eq:LS_pitchfork}
\end{equation}
Part A of the lemma follows by \eqref{eq:LS_pitchfork}, by the recognition problem for the pitchfork bifurcation \cite[Chapter II, Proposition 9.2]{Golubitsky1985}, as well as by the definition of a center manifold. Part B follows by the definition of an unfolding and by \eqref{eq:LS_pitchfork}.
\end{proof}

As a direct consequence of Theorem \ref{thm:pitchfork} we can describe many of the bifurcations of $\mathbf{x} = 0$ of \eqref{eq:dynamics_vectors} from the spectrum of  
$A$. In particular, if $A$ has $n \leq N_a$ simple eigenvalues $\lambda_i$, we expect $\mathbf{x} = 0$ to exhibit $n$ distinct pitchfork bifurcations at critical values of the parameter $u_i^* = d/(\alpha + \lambda_i \gamma)$. Locally near the bifurcation point the corresponding left eigenvector $\mathbf{v}_i$ informs the sign structure of the emergent equilibria, as explored in \cite{BizyaevaACC2020}. 
For undirected graphs we can deduce the direction in which the bifurcation branches appear.  

\begin{corollary}
Suppose $G$ is an undirected graph. When $u_i^* = d/(\alpha + \lambda_i \gamma) > 0 (<0)$ the pitchfork bifurcation at $u_i^*$ happens supercritically (subcritically). \label{cor:criticality}
\end{corollary}
\begin{proof}
Let $\mathbf{v}_{i}$ be the right eigenvector of $A$ corresponding to $\lambda_i$ and observe that for an undirected graph, the left eigenvector of $\lambda_i$ $\mathbf{w}_i = \mathbf{v}_i$ by symmetry. Then $\langle \mathbf{w}_i,\mathbf{v}_i\rangle = \langle \mathbf{v}_i,\mathbf{v}_i\rangle > 0$ and 
$\langle \mathbf{w}_i,\mathbf{v}_i^3\rangle = \langle \mathbf{v}_i,\mathbf{v}_i^3\rangle = \sum_{k = 1}^{N_a} (\mathbf{v}_{i})_{k}^4 > 0$. Therefore the criticality condition from Theorem \ref{thm:pitchfork} simplifies to $(\alpha + \lambda_i \gamma) > 0 (<0)$ for supercritical (subcritical) pitchfork bifurcation. Since $d>0$ the corollary directly follows. 
\end{proof}

Using these general results on the bifurcation behavior of the opinion dynamics, the next theorem establishes that the agreement and disagreement bifurcations  in Proposition \ref{propACC} are  supercritical pitchfork bifurcations in which $\mathbf{x} = 0$ loses stability and new branches of locally stable solutions appear. 

\begin{theorem}[Agreement and Disagreement Pitchforks]\label{thm: dis agree pitch}
Consider \eqref{eq:opinion_dynamics} and let $u_{i} := u \geq 0$. The agreement and disagreement bifurcations in Proposition \ref{propACC} are supercritical pitchfork bifurcations. Additionally, the two steady state solutions which appear for $u > u_a (u_d)$ and $|u - u_a| (|u - u_d|)$ sufficiently small are locally exponentially stable.

\label{thm:agree_disagree_pitchfork}
\end{theorem}
\begin{proof}



The supercriticality of the bifurcating branches of equilibria follows for the undirected case from Corollary \ref{cor:criticality}. For a directed graph and $\gamma > 0$ it follows from the Perron-Frobenius theorem that $\mathbf{v}_{max}$ and $\mathbf{w}_{max}$ have strictly positive components, i.e., $\langle \mathbf{w}_{max}, \mathbf{v}_{max} \rangle > 0$ and $\langle \mathbf{w}_{max}, \mathbf{v}^3_{max} \rangle > 0$. Supercriticality then follows from the conditions in Theorem \eqref{thm:pitchfork}. The two nontrivial fixed points are locally exponentially stable by analytic continuity of eigenvalues:  $N_a - 1$ negative eigenvalues are shared with $\mathbf{x} = 0$ and the bifurcating eigenvalue is negative by \cite[Chapter I, Theorem 4.1]{Golubitsky1985} because $\partial f/ \partial z  > 0$ for the Lyapunov-Schmidt reduction \eqref{eq:LS_pitchfork}.
\end{proof}

One major takeaway of the results presented in this section are rigorous predictions of the effect of inputs on the opinion formation bifurcation behavior. It follows by Theorem~\ref{thm:pitchfork}B and Theorem~\ref{thm: dis agree pitch} that the gain with which an input $b_i$ on node $i$ affects the network opinion formation behavior is exactly the node agreement or disagreement centrality $w_i$. In particular, this results allows us to predict in which direction the agreement or disagreement pitchfork unfold as a function of the input pattern. If $\langle \mathbf{b}, \mathbf{w}_{max} \rangle = 0$ the pitchfork does not unfold. If $\langle \mathbf{b}, \mathbf{w}_{max} \rangle < 0$ ($\langle \mathbf{b}, \mathbf{w}_{max} \rangle > 0$) the pitchfork unfolds in a such a way that it exhibits a lower (upper) smooth branch of equilibria.
For example, in Fig. \ref{fig:bif} the diagram on the left receives a nonzero input which is orthogonal to $\mathbf{w}_{max}$, and the symmetry of the pitchfork bifurcation is unbroken. On the right, $\langle \mathbf{b}, \mathbf{w}_{max} \rangle = 0.1$ and near the singular point of the symmetric diagram, the unfolded diagram favors the positive solution branch which corresponds to agents agreeing on the positive option.

\begin{figure} 
    \centering
    \includegraphics[width=0.9\linewidth]{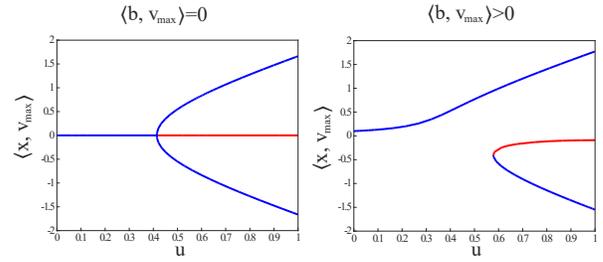}
    \caption{Symmetric pitchfork bifurcation and its unfolding for opinion dynamics \eqref{eq:opinion_dynamics} in the agreement regime  with three agents communicating over an undirected line graph. Blue (red) curves correspond to stable (unstable) equilibria. Vertical axis is the projection of the system equilibria onto $W({\lambda_{max}})$.  Parameters: $d=\alpha = \gamma = 1$. Left: $\mathbf{b} = (0.05,0,-0.05)$; right: $\mathbf{b} = 0.1 \mathbf{v}_{max} + (0.05,0,-0.05)$. Bifurcation diagrams  generated with help of MatCont 
    \cite{matcont}.}\label{fig:bif}
\end{figure}

\section{Dynamic Attention: Cascades and Tunable Sensitivity to Input}
\label{sec:dynamicu}

In this section we illustrate how distributed state feedback dynamics in the attention parameters of the opinion dynamics \eqref{eq:opinion_dynamics} gives rise to agreement and disagreement cascades with tunable sensitivity to distributed input. We show that the magnitude of the distributed input vector, and its orientation relative to the centrality eigenvector $\mathbf{v}_{max}$ ($\mathbf{v}_{min}$) when $\gamma > 0$ ($< 0$) can both be used as control parameters to trigger cascades on the network. We prove that a single design parameter in the attention feedback dynamics can be used to tune the threshold above which inputs trigger a cascade. 

As in \cite{Bizyaeva2020} we define state feedback dynamics for the attention parameter $u_{i}$ of each agent $i$ to track the saturated norm of the system opinion state observed by agent $i$:
\begin{equation}
    \tau_u \dot{u}_{i} = - u_{i} + S_{u} \left(x_{i}^{2} + \sum_{k = 1}^{N_a} (a_{ik} x_k)^2 \right ) \label{eq:u_dynamics}.
\end{equation}
$S_u$ takes the form of the Hill activation function:
\begin{equation}
    S_{u}(y) = \munderbar u + (\bar u - \munderbar u) \frac{y^{n}}{(y_{th})^{n} + y^{n}} \label{eq:Su}, 
\end{equation}
where threshold $y_{th} > 0$.  In \eqref{eq:Su} we constrain $\bar u$ and $\munderbar u$ such that  $\bar u > u_c \geq \munderbar u > 0$, 
with $u_c = u_a$ ($u_d$) when $\gamma>0$ ($<0$). As in \cite{Franci2021}, we define an \textit{opinion cascade} as a network transition from a weakly opinionated ($|x_i|\ll\vartheta$, for all $i$), 
weakly attentive ($|u_i|\ll\vartheta$, for all $i=1$) system state to a 
strongly opinionated ($|x_i|\gg\vartheta$, for all $i$) and strongly attentive ($|u_i|\gg\vartheta$, for all $i$) system state. See Fig. \ref{fig:cascade_traj} for an example of an opinion cascade 
in an agreement ($\gamma >0$) and disagreement ($\gamma < 0$) regime. In this section we restrict our attention to undirected graph, which largely simplify the relevant algebraic computations. 
\begin{assumption}
$G$ is undirected.
\end{assumption}

\begin{figure}
    \centering
    \includegraphics[width=0.9\linewidth]{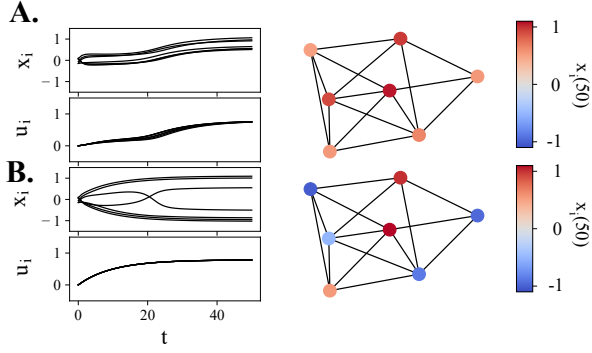}
    \caption{A. Agreement cascade, $\gamma = 1$, $ \underline u = u_a - 0.01$, $\bar u = u_a + 0.6$; B. Disagreement cascade, $\gamma = -1$, $\underline u = u_d - 0.01$, $\bar u = u_d + 0.6$. Left) Opinion and attention trajectories over time $t$. Right) Graph structure. Color of each node $i$ is opinion $x_i(t)$ at $t=50$. Parameters: $d = 1$, $n = 3$, $u_{th} = 0.4$, $\tau_u = 10$, $\alpha = 1$, $d = 1$. Each $x_i(0)$  drawn from $\mathcal{N}(0,0.1)$ and each $u_i(0) = 0$. Input vector was randomly generated: each $b_i$ drawn from $\mathcal{N}(0,0.2)$.  }
    \label{fig:cascade_traj}
\end{figure}

In vector form, coupled dynamics \eqref{eq:opinion_dynamics},\eqref{eq:u_dynamics} read
\begin{equation}
    \begin{pmatrix} \dot{\mathbf{x}} \\ \dot{\mathbf{u}} \end{pmatrix}= - \begin{pmatrix}d \mathbf{x} \\ \mathbf{u} \end{pmatrix} +  \begin{pmatrix}\mathbf{u} \odot \mathbf{S}\left( (\alpha \mathcal{I} + \gamma A) \mathbf{x}\right) + \mathbf{b}\\
      \mathbf{S}_{u}( (\mathcal{I}+A) \mathbf{x}^{2} ) \end{pmatrix}
     \label{eq:coupled_dynamics_vectors}
\end{equation}
where  $\mathbf{S}_u(\mathbf{y}) = (S_u(y_{1}), \dots, S_u(y_{n}))$, $\mathbf{y} \in \mathds{R}^{n}$, $\mathbf{x}^{2} = (x_1^{2}, \dots, x_{N_{a}}^{2}) $, and $\odot$ is the element-wise product of vectors. The Jacobian of~\eqref{eq:coupled_dynamics_vectors}  at  equilibrium point $(\mathbf{x}_{s},\mathbf{u}_{s})$ is 
\begin{equation}
    J_{(\mathbf{x},\mathbf{u})} = \begin{pmatrix} -d \mathcal{I} +  \left(\operatorname{diag} \{\mathbf{u}_s\} (\alpha \mathcal{I} + \gamma A) \right)\odot K_{1}  & K_{2}  \\ 
     (\mathcal{I} + A) \operatorname{diag}\{\mathbf{x}_s\} \odot K_{3} & - \mathcal{I} \end{pmatrix} \label{eq:jac_coupled}
\end{equation}
where 
$K_{1} = \mathbf{S}'\left( (\alpha \mathcal{I} + \gamma A) \mathbf{x}_{s}\right) \mathbf{1}^{T}$, $K_{2} = \operatorname{diag} \{ \mathbf{S}\left( (\alpha \mathcal{I} + \gamma A) \mathbf{x}_{s}\right) \}$,  $K_{3} = 2\mathbf{S}_u'\left( (\mathcal{I} + A) \mathbf{x}^{2}_{s}\right) \mathbf{1}^{T}$, and $\mathbf{1} = (1, \dots, 1) \in \mathds{R}^{N_a}$. Let $G(\mathbf{y},\mathbf{b})$ be the right hand side of \eqref{eq:coupled_dynamics_vectors} 
with $\mathbf{y} = (\mathbf{x},\mathbf{u})$. 

\begin{lemma}[Stability of $\mathbf{x} = 0$]
Consider \eqref{eq:coupled_dynamics_vectors} with $\mathbf{b} = 0$. The point $(\mathbf{x}_s,\mathbf{u}_s) = (\mathbf{0},\munderbar u\mathbf{1})$ is an equilibrium point of the coupled dynamics. When either $\gamma > 0$ and $\munderbar u < u_{a}$ or  $\gamma < 0$ and $\munderbar u < u_{d}$, it is locally exponentially stable. \label{lem:stability_origin}
\end{lemma}
\begin{proof}
Plugging the state values into the coupled dynamics \eqref{eq:coupled_dynamics_vectors} easily verifies that $(\dot{\mathbf{x}},\dot{\mathbf{u}}) = 0$ at $(\mathbf{x}_s,\mathbf{u}_s)$. Evaluated at this point, \eqref{eq:jac_coupled} simplifies to the block diagonal matrix
\begin{equation}
    J_{(0,\munderbar u)} = \begin{pmatrix} -d \mathcal{I} + \munderbar u  (\alpha \mathcal{I} + \gamma A)   & 0  \\ 
     0 & - \mathcal{I} \end{pmatrix}. \label{eq:jac_diag}
\end{equation}
When $0 < \munderbar u < u_a$ ($u_d$), \eqref{eq:jac_diag} has $2 N_{a}$ eigenvalues with negative real part, and  the stability conclusion follows.
\end{proof}

\begin{lemma}[Small Input Approximates Equilibrium Opinion]\label{lem:small input approx}
Let $(\mathbf{x}_{s}, \mathbf{u}_{s})$ be an equilibrium of \eqref{eq:coupled_dynamics_vectors} with inputs $\mathbf{b}$. Let $\munderbar u < u_c$ where $u_c = u_a$ if $\gamma > 0$ or $u_{c} = u_d$ if $\gamma < 0$. Define $\mathbf{v} =\mathbf{v}_{max}$ if $\gamma > 0$ or  $\mathbf{v} =\mathbf{v}_{min}$ if $\gamma < 0$.   Then \label{lem:xeq}
\begin{equation}
    \left.\frac{\partial \| \mathbf{x}_s \|}{\partial | \langle \mathbf{v}, \mathbf{b} \rangle |}\right|_{\mathbf{b}=\mathbf{x}_s=0} > 0. 
\end{equation}
\end{lemma}
\begin{proof}
Since $\mathbf{x} = 0$ is an equilibrium of the system with $\mathbf{b} = 0$, $(\mathbf{x}_{s}, \mathbf{u}_{s})$ can be approximated by the linearization
\begin{equation}
    J_{(0,\munderbar u)}\begin{pmatrix} \mathbf{x}_{s} \\ \mathbf{u}_{s} \end{pmatrix} + \begin{pmatrix} \mathbf{b} \\ \mathbf{0}
    \end{pmatrix} = \mathbf{0}.
\end{equation}
$J_{(0,\munderbar u)}$ is symmetric and invertible so its inverse has the same eigenvectors. 
Then, to the first order, it holds that
\begin{equation}
    \mathbf{x}_{s} = - J_x^{-1} \mathbf{b} = - \sum_{i = 1}^{N_a} \frac{1}{\mu_{i}}  \langle \mathbf{v}_i, \mathbf{b} \rangle \mathbf{v}_{i} 
\end{equation}
where $J_x$ is \eqref{eq:jac} with $u = \munderbar u$, $\mathbf{v}_{i}$ is an eigenvector of $A$ corresponding to eigenvalue $\lambda_i$, and $\mu_i = d + \munderbar u(\alpha + \lambda_i \gamma)$. The eigenvectors are orthogonal, so $\| \mathbf{x}_s \| = \sum_{i = 1}^{N_a}\frac{1}{\lambda_{i}^2}\langle \mathbf{v}_i, \mathbf{b} \rangle^2$ and the lemma follows.
\end{proof}

\begin{theorem}[Saddle-Node Bifurcation]
Consider \eqref{eq:coupled_dynamics_vectors} with a nonzero input vector $\mathbf{b}$ and define $u_{c} = u_a$ if $\gamma > 0$ and $u_c = u_d$ if $\gamma < 0$. Let $\mathbf{v}_c = \mathbf{v}_{max}$ or $\mathbf{v}_{min}$ respectively. Suppose $u_{th} \ll 1$ and $\munderbar u < u_{c}$ with $|\munderbar u - u_c|$ sufficiently small. There exists $p>0$ such that when $| \langle \mathbf{v}_c, \mathbf{b} \rangle | = p$ there exists an equilibrium $(\mathbf{x}_p, \mathbf{u}_p)$ of  ~\eqref{eq:coupled_dynamics_vectors} such that, if
\begin{equation}
    \langle \mathbf{v}_c, \mathbf{u}_p \odot \mathbf{v}_c \odot \mathbf{v}_i \odot \mathbf{S}''((\alpha \mathcal{I} + \gamma A) \mathbf{x}_p) \rangle < 0 \label{eq:IP_condition}
\end{equation}
is verified at $(\mathbf{x}_p, \mathbf{u}_p)$ with $\lambda_i \neq \lambda_c$ an eigenvalue of $A$ with corresponding eigenvector  $\mathbf{v}_i$ :
(i) There exists a smooth curve of equilibria in $\mathds{R}^{2 N_{a}} \times \mathds{R}$ passing through $(\mathbf{x}_p, \mathbf{u}_p,p)$, tangent to the hyperplane $\mathds{R}^{2 N_a} \times \{p\}$; (ii). There are no equilibria near $(\mathbf{x}_p, \mathbf{u}_p,p)$ when $ |\langle \mathbf{v}_c, \mathbf{b} \rangle | > p$ and two equilibria near $(\mathbf{x}_p, \mathbf{u}_p,p)$ for each $ |\langle \mathbf{v}_c, \mathbf{b} \rangle | < p$; (iii). The two equilibria near  $(\mathbf{x}_p, \mathbf{u}_p,p)$ are hyperbolic and have stable manifolds of dimensions $N_a$ and $N_a - 1$ respectively. \label{thm:saddle}
\end{theorem}

\begin{proof}

Observe that \eqref{eq:jac_coupled} depends continuously on the model parameters and on the state. Then, by \cite[Chapter II, Theorem 5.1]{kato2013perturbation} the eigenvalues and eigenvectors of \eqref{eq:jac_coupled} changes continuously for $\| \mathbf{x}_{s}\|$ sufficiently small. Due to space constraints, we leave the full development of the matrix perturbation theory for future work and instead conjecture that if $\| \mathbf{x}_s\|$ is sufficiently small then the eigenvectors of \eqref{eq:jac_diag} are a good approximation of the eigenvectors of \eqref{eq:jac_coupled}. Since the origin of \eqref{eq:coupled_dynamics_vectors} with $\mathbf{b} = 0$ is stable by Lemma \ref{lem:stability_origin} and because $\lambda_{min}$ and $\lambda_{max}$ are simple eigenvalues, if an eigenvalue of $J_{(\mathbf{x}_s,\mathbf{u}_s)}$ crosses zero for some $\| \mathbf{b} \|$ it must also be simple. Furthermore this eigenvalue corresponds to a perturbation of $-d + \munderbar u(\alpha + \gamma \lambda_{c})$ where $\lambda_c = \lambda_{max}$ or $\lambda_{min}$ respectively.

By Lemma~\ref{lem:small input approx}, if $\mathbf{b}\neq 0$ then at equilibrium $\|\mathbf{x}\| \neq 0$.
Let $g(z,\mathbf{b})$ be the Lyapunov-Schmidt reduction of \eqref{eq:coupled_dynamics_vectors} at an equilibrium $(\mathbf{x}, \mathbf{u})$ for sufficiently small inputs. We have
\begin{multline}
    d^{2}G_{\mathbf{y}_{p},\mathbf{b}_{p}}(\Tilde{\mathbf{v}}_c,\Tilde{\mathbf{v}}_c) = \sum_{j = 1 }^{N_a} \sum_{k = 1}^{N_a}  \left.\frac{\partial^2 (G)_i}{\partial x_j \partial x_k} (\mathbf{v}_c)_j (\mathbf{v}_c)_k\right|_{(\mathbf{y}_{p},\mathbf{b}_{p})}\\
    = (\alpha + \lambda_c \gamma)^2 \begin{pmatrix}\mathbf{u}_{p} \\ \mathbf{0} \end{pmatrix} \odot \begin{pmatrix} \mathbf{S}''((\alpha \mathcal{I} + \gamma A) \mathbf{x}_{p}) \\ \mathbf{0} \end{pmatrix}. 
\end{multline}
and the second derivative in the Lyapunov-Schmidt reduction evaluates to
\begin{multline*}
    g_{zz} = \langle \Tilde{\mathbf{v}}_c,  d^{2}G_{\mathbf{y}_{p},\mathbf{b}_{p}}(\Tilde{\mathbf{v}}_c,\Tilde{\mathbf{v}}_c) \rangle\\
    = (\alpha + \lambda_c \gamma)^2  \sum_{i = 1}^{N_a} (\mathbf{u}_p)_i (\mathbf{v}_{c})_i S''\left(\alpha (\mathbf{x}_{p})_{i} + \gamma \sum_{\substack{k=1 \\ k \neq i}}^\Na a_{ik} (\mathbf{x}_{p})_{k}\right) > 0
\end{multline*}
where we used Lemma~\ref{lem:small input approx} to approximate $\mathbf{x}_p \approx k \mathbf{v}_c$ for some $k \in \mathds{R}$ and we can choose $\mathbf{v}_c$ for which $k > 0$ and such that
\begin{multline*}
     \operatorname{sign} \{ (\mathbf{v}_{c})_{i} \} =  \operatorname{sign} \left\{S''\left(\alpha (\mathbf{x}_{p})_{i} + \gamma \sum_{\substack{k=1 \\ k \neq i}}^\Na a_{ik} (\mathbf{x}_{p})_{k}\right) \right\} \\
     = \operatorname{sign} S''\left( k(\alpha + \lambda_c \gamma) (\mathbf{v}_{c})_{i} \right)
\end{multline*}
Additionally, 
\begin{equation*}
    g_{b_i} = \left\langle \Tilde{\mathbf{v}}_c, \frac{\partial G}{ \partial b_{i}} \right\rangle = (\Tilde{\mathbf{v}}_c)_i
\end{equation*}
which means the term $\langle \mathbf{v}_c, \mathbf{b} \rangle$ appears in $g(z,p)$. 

Finally, we compute the coefficient of the cross-term $g_{z \hat{b}}$ in the Lyapunov-Schmidt reduction. For computational convenience, we express the input vector as $\mathbf{b} = \sum_{i = 1}^{N_a} \beta_i \mathbf{v}_i$ where each $\beta_i := \langle \mathbf{v}_i, \mathbf{b}\rangle$. Coefficients of the cross-terms $z \beta_i$ in $g(z, \mathbf{b})$ simplify to

\begin{multline}
    g_{z \beta_i} = \left\langle \Tilde{\mathbf{v}}_c, -d^2 G_{\mathbf{y}_{p},\mathbf{b}_{p}} \left(\Tilde{\mathbf{v}}_c, J_{(0,u)}^{-1} E \left( \frac{\partial G}{\partial \beta_i}\right) \right)\right\rangle
\end{multline}
where $E$ is a projection onto the range of $J_{(0,\munderbar u)}$ and $J_{(0,\munderbar u)}^{-1}: \mathbf{v}_c^{\perp} \mapsto \mathds{R}^{N_a}$ is the inverse of to restriction of  $J_{(0,\munderbar u)}$ to the orthogonal complement to $\mathbf{v}_c$. We find that $J_{(0,u)}^{-1} E \left( \frac{\partial G}{\partial \beta_i}\right) = \frac{1}{\mu_i} (\mathbf{v}_i,\mathbf{0})$ and 
$ g_{z \beta_i} = -(\alpha + \lambda_c \gamma) K_i$
where each $K_i$ is the quantity in \eqref{eq:IP_condition}. Since $g_{z \beta_i} > 0$ for all $i$, we conclude that the eigenvalue of the equilibrium is monotonically increasing with $|<\mathbf{v}_c,\mathbf{b}>|$. By continuity of eigenvalues of the pertubred Jacobian,
it follows that the leading eigenvalue necessarily crosses zero as input is increased.
By \cite[Theorem 3.4.1]{guckenheimer2013nonlinear} this singularity must be a saddle-node bifurcation point, with bifurcation parameter $\hat{b} = \langle \mathbf{v}_c, \mathbf{b} \rangle$ and properties outlined in the statement of the theorem.

\end{proof}

\begin{corollary}\label{cor: cascade dual control}
The input magnitude $\| \mathbf{b} \|$ and its relative orientation $\mathbf{b} \angle \mathbf{v}_c := \langle \mathbf{v}_c, \mathbf{b} \rangle/\| \mathbf{b} \|$ can be both used as controls to trigger a network opinion cascade.
\end{corollary}
\begin{proof}
This follows trivially by factoring out the magnitude of the input vector from the bifurcation parameter $\langle \mathbf{v}_c, \mathbf{b} \rangle$. 
\end{proof}
Figure~\ref{fig:coupled_bif} illustrates the prediction of Corollary~\ref{cor: cascade dual control}. In particular, the figure shows bifurcation diagrams with stable and unstable equilibria of the opinion dynamics in the agreement regime on a small network. The two diagrams illustrate how with $\| \mathbf{b} \|$ and $\mathbf{b} \angle \mathbf{v}$ as bifurcation parameters, the saddle-node bifurcation predicted by Theorem \ref{thm:saddle} is observed. Opinion cascades are activated when the bifurcation parameter passes the critical value. Although the predictions of the results in this section assume inputs are small, in simulation and through numerical continuation of the dynamics on different networks we observe that this result is quite robust. The existence of a saddle-node bifurcation, and therefore a threshold which differentiates between inputs which trigger a cascade and ones which do not, persists across network structures and for large inputs.

\begin{figure}
    \centering
    \includegraphics[width=0.9\linewidth]{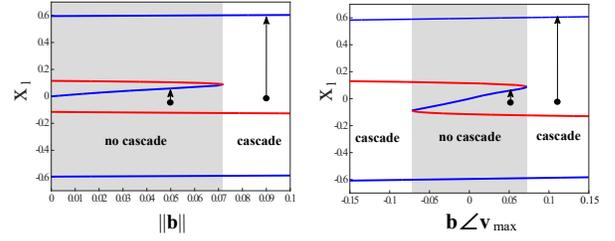}
    \caption{Blue (red) lines track the first coordinate of the stable (unstable) equilibrium solutions of the coupled dynamics \eqref{eq:coupled_dynamics_vectors} on a 3-agent undirected line graph. Parameters: $u_{th} = 0.1$, $\gamma = 1$, $n = 3$, $\mathbf{b} = \|\mathbf{b}\| \cdot |\mathbf{b} \angle \mathbf{v}_{max} | \cdot \mathbf{v}_{max} $; left: $|\mathbf{b} \angle \mathbf{v}_{max} | = 0.1$, right: $\|\mathbf{b}\| = 0.1$  Bifurcation diagrams  generated  using MatCont 
    \cite{matcont}.}
    \label{fig:coupled_bif}
\end{figure}



\begin{figure}
    \centering
    \includegraphics[width=0.9\linewidth]{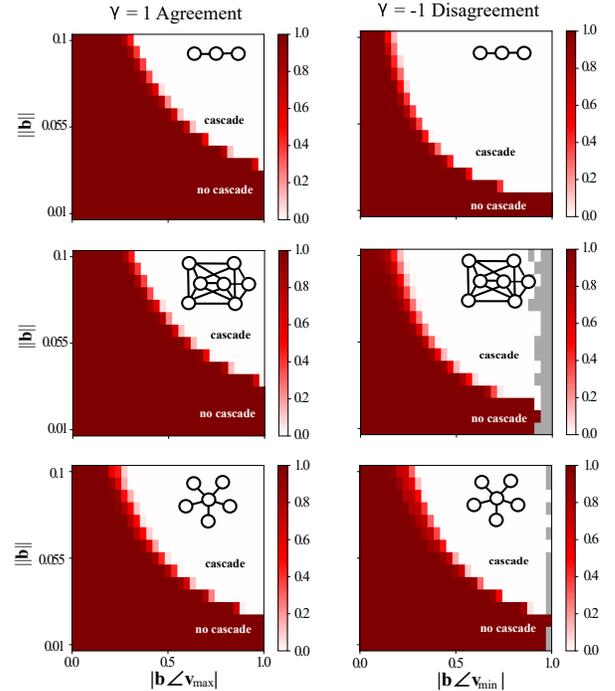}
    \caption{Heatmaps with color corresponding to proportion of simulations within the given parameter range which did not result in a network cascade by $t=500$. Dark red corresponds to no cascades, white corresponds to there always being a cascade, and grey squares are bins with no datapoints. Each plot corresponds to $1.5\times 10^5$ distinct simulations on an undirected graph shown in the diagram. Simulation parameters: $\tau_u = 10$, $u_{th} = 0.2$, $\underline u = u_{a} - 0.01$ for $\gamma = 1$ (left plots) and  $\underline u = u_{d} - 0.01$ for $\gamma = -1$ (right plots). For each simulation, inputs $b_i$ were drawn from $\mathcal{N}(0,1)$ and the input vector $\mathbf{b}$ was normalized to a desired magnitude. There were 10000 simulations performed at each constant input magnitude, with 15 magnitudes sampled uniformly spaced between $0$ and $0.1$. The initial conditions for each simulation were $x_i = 0$, $u_i = 0$ for all $i = 1, \dots, N_a$. }
    \label{fig:cascade_sims}
\end{figure}

A consequence of Theorem~\ref{thm:saddle} is that also for opinion cascades the node centrality indices are the key determinant of the effect of exogenous inputs on the coupled attention-opinion dynamics~\eqref{eq:coupled_dynamics_vectors}. More precisely, the smaller the angle between the input vector and the agreement or disagreement centrality vector, the smaller the needed input strength to trigger an agreement or a disagreement opinion cascade. Figure~\ref{fig:cascade_sims} numerically illustrates our theoretical prediction. The transition line from the red (no cascade) to the white (cascade) regions correspond to the threshold, i.e., the saddle node bifurcation predicted by Theorem~\ref{thm:saddle}, at which the opinion cascade is ignited. For disparate network topologies and both for agreement and disagreement opinion cascades, it shows that as the angle between the input vector and the centrality vector decreases, the norm of the input needed to trigger a cascade gets smaller. 

For the simulations shown in Figure \ref{fig:cascade_sims}, in the cascade region of the simulations we confirmed numerically that the centrality eigenvector accurately predicts the sign distribution among the nodes. Rigorously proving this is subject of future work. Although we do not explore it in this paper, we additionally note that the cascade threshold is implicitly defined by the design parameter $y_{th}$ in the attention saturation function \eqref{eq:Su}. In future work we will explore how the sensitivity of the group to distributed input can be tuned with this parameter. Other future directions include expanding the analysis presented here to multi-option cascades with the general opinion dynamics model in \cite{Bizyaeva2020} and to connect it to other continuous time and state-space models of opinion cascades such as \cite{zhong2019continuous}. 



\bibliographystyle{./bibliography/IEEEtran}
\bibliography{./bibliography/references}

\end{document}